\documentclass[11pt,textwidth=125mm,pagedepth=185mm]{article}

\usepackage{authblk}

\usepackage{amssymb,amsfonts}
\usepackage{amsmath,amsthm}
\usepackage{enumerate}
\usepackage[english]{babel}
\usepackage{amscd}
\usepackage{pgf,tikz}
\usepackage[utf8]{inputenc}

\usepackage[left=2cm,right=2cm, top=3cm, bottom=3cm]{geometry}

\usepackage[all,cmtip]{xy}
\usepackage{enumerate}
\xyoption{all}
\usepackage{srcltx}
\usepackage{textcomp}
\usepackage{parskip}

\usepackage{amsthm, amsmath, amssymb,latexsym} 
\usepackage{amsfonts,epsfig,amscd}
\usepackage[]{fontenc}
\usepackage{xy}
\usepackage{enumerate}
\usepackage[]{fontenc}
\usepackage[all]{xy}

\newtheorem*{theorem*}{Theorem}
\newtheorem{theorem}{Theorem}[section]
\newtheorem{lemma}[theorem]{Lemma}
\newtheorem{corollary}[theorem]{Corollary}
\newtheorem{proposition}[theorem]{Proposition}
\newtheorem{remark}[theorem]{Remark}
\newtheorem{definition}[theorem]{Definition}

\newcommand{\ra}{\rightarrow}

\newcommand{\longra}{\longrightarrow}

\newcommand{\Longlra}{\Longleftrightarrow}

\newcommand{\hra}{\hookrightarrow}

\newcommand{\beq}{\begin{equation}}
\newcommand{\enq}{\end{equation}}
\newcommand{\beqn}{\begin{equation*}}
\newcommand{\enqn}{\end{equation*}}

\newcommand{\caA}{\mathcal{A}}

\newcommand{\caC}{\mathcal{C}}

\newcommand{\caH}{\mathcal{H}}

\newcommand{\caK}{\mathcal{K}}

\newcommand{\caM}{\mathcal{M}}

\newcommand{\caO}{\mathcal{O}}

\newcommand{\caT}{\mathcal{T}}
\newcommand{\caU}{\mathcal{U}}

\newcommand{\mC}{\mathbb{C}}

\newcommand{\mH}{\mathbb{H}}
\newcommand{\mN}{\mathbb{N}}
\newcommand{\mP}{\mathbb{P}}

\newcommand{\mR}{\mathbb{R}}
\newcommand{\mU}{\mathbb{U}}
\newcommand{\mV}{\mathbb{V}}
\newcommand{\mZ}{\mathbb{Z}}

\DeclareMathOperator{\Hom}{\rm Hom}

\DeclareMathOperator{\Sym}{\rm Sym}

\DeclareMathOperator{\ic}{\rm int}
\DeclareMathOperator{\id}{\rm id}
\DeclareMathOperator{\im}{\rm im}

\DeclareMathOperator{\rk}{\rm rk}
\DeclareMathOperator{\supp}{\rm supp}

\DeclareMathOperator{\nGM}{\nabla^{GM}}
\DeclareMathOperator{\nE}{\nabla^{E}}
\DeclareMathOperator{\nEd}{\nabla^{E^{\vee}}}
\DeclareMathOperator{\nH}{\nabla^{T}}

\newcommand{\nGMapp}[1]{\nabla^{GM}_{#1}}
\newcommand{\nEapp}[1]{\nabla^{E}_{#1}}
\newcommand{\nEdapp}[1]{\nabla^{E^{\vee}}_{#1}}
\newcommand{\nHapp}[1]{\nabla^{T}_{#1}}
\newcommand{\abl}[1]{\frac{\partial}{\partial {#1}}}
\newcommand{\ablt}[2]{\frac{\partial {#1}}{\partial {#2}}}
\newcommand{\contr}[1]{\ic_{#1}}

\title{Punctual characterization of the unitary flat bundle of weight $1$ PVHS and application to families of curves}
\makeatletter

\@addtoreset{equation}{section}
\makeatother

\author[1]{V\'ictor Gonz\'alez-Alonso \thanks{gonzalez@math.uni-hannover.de, corresponding author}}
\author[1]{Sara Torelli \thanks{torelli@math.uni-hannover.de}}
\affil[1]{Leibniz Universität Hannover, Welfengarten 1, 30167 Hannover}

\date{\today}

\begin{document}	
	
\maketitle
	
\begin{abstract}
In this paper we consider the problem of pointwise determining the fibres of the flat unitary subbundle of a PVHS of weight one. Starting from the associated Higgs field, and assuming the base has dimension $1$, we construct a family of (smooth but possibly non-holomorphic) morphisms of vector bundles with the property that the intersection of their kernels at a general point is the fibre of the flat subbundle. We explore the first one of these morphisms in the case of a geometric PVHS arising from a family of smooth projective curves, showing that it acts as the cup-product with some sort of ``second-order Kodaira-Spencer class'' which we introduce, and check in the case of a family of smooth plane curves that this additional condition is non-trivial.
\end{abstract}

\bigskip

\section{Introduction}

Consider a polarized variation of Hodge structures (PVHS) of weight one and rank $2g$ on a smooth complex manifold $B$, which in particular consists of a short exact sequence of holomorphic vector bundles on $B$
$$0\longra E=E^{1,0}\longra \caH=\mV_{\mZ}\otimes_{\mZ}\caO_B\longra E^{0,1}\longra 0$$
together with a Gauß-Manin connection on $\caH$. The bundle $E$ carries a natural (maximal) flat unitary subbundle $\caU\subseteq E$ which encodes many properties of the natural modular map $B\ra\caA_g$ and its relation to the (open) Torelli locus $\caT_g\subseteq\caA_g$ of jacobian varieties.

In particular, it can be used to study the Coleman-Oort conjecture on the (non)-existence of Shimura curves generically contained in $\caT_g$. For example, recently Lu and Zuo used some properties of $\caU$ in \cite{LZ-Oort-Torelli} to prove that certain types of Shimura curves do not lie in $\caT_g$. Moreover, in \cite{CLZ16,CLZ18} together with Chen they proved that for too big or too small rank of $\caU$, the map $B\ra\caA_g$ does not describe an open subset of a Shimura curve. In a slightly more general setting, the flat unitary bundle $\caU$ has been proven a very useful tool in the recent work of the second author with Ghigi and Pirola \cite{GPT19} to bound the dimension of the totally geodesic subvarieties of $\caT_g\subseteq\caA_g$.

In the geometric case, where the PVHS arises from a family of smooth projective complex curves of genus $g$ over a quasi-projective curve $B$, $\caU$ arises from the second Fujita decomposition \cite{Fuj78b,CD:Answer_2017,Catanese-Kawamata}. Its rank is an important invariant that can be used to study Xiao's conjecture on the relative irregularity (see the recent papers of the authors in collaboration with Barja, Naranjo and Stoppino \cite{BGN_Xiao_2015}, \cite{GST17}) and has also applications to the study of hyperelliptic fibrations \cite{LZ17}.

It is therefore very interesting to have tools to compute the rank of $\caU$ as explicitly as possible. However, since evaluating the Gauß-Manin connection involves knowing the sections of $E$ in open subsets, no direct characterization of the fibres $\caU_b\subseteq E_b$ depending only on the point $b$ is yet known, even if we assume $b\in B$ to be general. This makes the computation or even the estimation of the rank of $\caU$ a difficult problem, which in the geometric case can be done directly only under special circumstances (e.g. for families of cyclic covers of $\mP^1$, or if the family of curves is supported in relatively rigid divisors, as developed by the first author in \cite{Gonz_OnDef_2016}). 

In the general case the first approach is by somehow ``linearising'' the Gauß-Manin connection on the fibres: considering the associated Higgs field $\theta\colon E\ra E^{0,1}\otimes\Omega_B^1$, which by definition satisfies $\caU_b\subseteq\ker\theta_b$ for any $b\in B$ (see \eqref{eq:df-Higgs} and Definition \ref{df:Subbundles}). The Higgs field is actually a homomorphism of holomorphic vector bundles, and $\theta_b$ depends only of infinitesimal data around the point $b$. For example, in the geometric case, $\theta_b$ is determined by the cup-product with the Kodaira-Spencer class of the infinitesimal deformation of the corresponding fibre. Thus studying $\ker\theta_b$ and in particular computing its rank is a simpler task, and in the geometric case several techniques have been developed by the authors in \cite{BGN_Xiao_2015,GST17}.

Ghigi, Pirola and the second author prove that, if $B\subseteq\caA_g$ contains one (real) geodesic curve, then $\caU=\caK:=\ker\theta$. On the other hand, in the recent work \cite{GT20}, the authors showed that $\caU$ and $\caK$ can be arbitrarily different also in the geometric case, so the results obtained by bounding the rank of $\caK$ can be very far away from the actual rank of $\caU$. Our aim in the present work is to construct new linear conditions on the fibres $E_b$ that define precisely $\caU_b\subseteq E_b$, at least for a general $b\in B$.

The main results of the paper are the following two theorems.

\begin{theorem*}[\ref{thm:U-with-etas}]
Let $\dim B=1$. Then there are smooth morphisms of vector bundles
$$\eta^{\left(1\right)},\eta^{\left(2\right)},\ldots,\eta^{\left(g\right)}\colon E\ra E^{0,1}$$
such that for any $\alpha\in\Gamma\left(E\right)$ it holds
\begin{equation*}
\alpha\in\Gamma\left(\caU\right)\Longlra\eta^{\left(1\right)}\left(\alpha\right)=\eta^{\left(2\right)}\left(\alpha\right)=\ldots=\eta^{\left(g\right)}\left(\alpha\right)=0.
\end{equation*}
In particular we have
\begin{equation*}
\caU_b\subseteq\bigcap_{k=1}^{g}\ker\eta^{\left(k\right)}_b\subseteq E_b
\end{equation*}
with equality for $b$ in a dense Zariski-open subset of $B$.
\end{theorem*}

\begin{theorem*}[\ref{thm:mu-hat}]
Let $f\colon \caC\ra B$ be a family of smooth projective curves $C_b=f^{-1}\left(b\right)$ with $\dim B=1$. For any $b\in B$ let $\caK_b\subseteq E_b=H^0\left(\omega_{C_b}\right)$ be the fibre of $\caK$ on $B$, and $\mu_b\in H^1\left(T_{C_b}\right)$ the second-order Kodaira-Spencer class of $C_b\subseteq\caC$ (Definition \ref{df:mu}). Let
$$\hat{\mu_b}\colon H^0\left(\omega_{C_b}\right)\stackrel{\mu_b\cdot}{\longra}H^1\left(\caO_{C_b}\right)=E_b^{\vee}\twoheadrightarrow\caK_b^{\vee}.$$
Then $\caU_b\subseteq\caK_b\cap\ker\hat{\mu_b}$.
\end{theorem*}

It is worth to note that the additional morphisms of vector bundles in Theorem \ref{thm:U-with-etas} are constructed from the Higgs field $\theta$ by applying the connection $\nH$ induced on $\Hom^s\left(E,E^{0,1}\right)$ by the Gauß-Manin connection. If the base $B$ of the family is actually a submanifold of $\caA_g$ (or the Siegel upper-half space $\mH_g$), this connection $\nH$ is the restriction to $B$ of the Levi-Civita connection of the Siegel metric. Thus our construction hints again at a connection between the equality $\caU=\caK$ and the existence of geodesics inside $B$ with respect to the Siegel metric.

Theorem \ref{thm:mu-hat} is a partial characterization of $\caU_b\subseteq\caK_b$ in the geometric case given by the cup-product with a cohomology class $\mu_b\in H^1\left(T_{C_b}\right)$ arising from $\eta^{\left(2\right)}$, which we call \emph{second-order Kodaira-spencer class} of the fibre $C_b$ in $\caC$ because it depends on the second-order infinitesimal neighbourhood.

The paper is organized as follows. In Section 2 we set up the main notations for PVHS, find an intermediate characterization of the sections of $\caU$ still depending on the Gauß-Manin connection (Theorem \ref{thm:U-with-nablas}) and we then define the morphisms of vector bundles $\eta^{\left(k\right)}$ (Definition \ref{df:etas}) that eventually lead to the punctual characterization of $\caU_b$ in Theorem \ref{thm:U-with-etas}.

In Section 3 we adapt our set up to the geometric case and study deeper the homomorphism $\eta^{\left(2\right)}$. We compute an explicit coordinate description that leads us to the Definition of the second-order Kodaira-Spencer class $\mu$ in Definition \ref{df:mu}. We finally show that, after projecting to $\caK_b^{\vee}$, both $\mu$ and $\eta^{\left(2\right)}$ impose the same condition on $\caK_b$ and obtain thus Theorem \ref{thm:mu-hat}.

In the final section 4 we implement the definition of $\mu$ for an arbitrary family of smooth plane curves, and show that $\mu$ actually is in general independent of the first-order infinitesimal deformation (Theorem \ref{thm:mu-plane}), hence the additional restriction of Theorem \ref{thm:mu-hat} is non-trivial.

{\bf Acknowledgements:} The authors would like to thank P. Frediani, A. Ghigi, L. Stoppino and G.P. Pirola for very useful and enlightening conversations around the topic.

\section{Characterizing the sections of the flat unitary subbundle}

In this section we first of all set up the notation we will use for PVHS of weight $1$ and we recall the definitions and some  basic properties of two naturally associated objects: the flat unitary bundle $\caU$ and the kernel sheaf $\caK$. Then we move on to the study of such objects. More precisely we aim to get a better understanding of the natural inclusion $\caU\subseteq \caK$, which so far has been characterized only by conditions of local kind. We present here our new results that allow to cut out $\caU\subseteq \caK$ just by using linear punctual conditions.

We will denote by $B$ a smooth complex manifold, with holomorphic tangent and cotangent bundles $T_B$ and $\Omega_B^1$. When necessary we will denote by $T_{B,\mR}$ and $T_{B,\mC}:=T_{B,\mR}\otimes_{\mR}\mC$ the real and complexified tangent bundles. Soon we will assume that $\dim B=1$, but the first definitions can be stated for any dimension. We identify any holomorphic vector bundle $F$ on $B$ with its sheaf of holomorphic sections. We will denote by $\caA\left(F\right)=F\otimes_{\caO_B}\caC^{\infty}\left(B\right)$ the sheaf of smooth sections of $F$, and more generally by $\caA^k\left(F\right)$ (resp. $\caA^{p,q}\left(F\right)$) the sheaf of smooth $k$-forms (resp. $\left(p,q\right)$-forms) with values in $F$.

We consider a polarized variation of Hodge structures (PVHS) of weight one $\left(\mV_{\mZ},E,Q\right)$. Here $\mV_{\mZ}$ is a local system of free abelian groups of rank $2g$, $Q\colon \mV_{\mZ}\otimes_{\mZ}\mV_{\mZ}\ra\mZ$ is a polarization, i.e. a non-degenerate antisymmetric unimodular $\mZ$-bilinear map, and $E=E^{1,0}\subseteq \caH:=\mV_{\mZ}\otimes_{\mZ}\caO_B$ is a holomorphic subbundle of rank $g$ where the Hodge metric $h\left(u,v\right)=iQ_{\mC}\left(u,\overline{v}\right)$ is positive definite. In particular it holds $E^{\perp}=\overline{E}$ with respect to $h$, and hence there is a decomposition $E\oplus\overline{E}=\caH$ as $\caC^{\infty}$ vector bundles on $B$. Moreover $Q$ (or rather $h$) induces a holomorphic $\mC$-linear isomorphism $E^{\vee}\cong E^{0,1}:=\caH/E$, hence there is the short exact sequence of holomorphic vector bundles
\begin{equation} \label{eq:PVHS}
0 \longra E \stackrel{\iota_1}{\longra} \caH=\mV\otimes_{\mC}\caO_B \stackrel{\pi_2}{\longra} E^{\vee}\longra 0,
\end{equation}
and a the orthogonal direct sum decomposition $\caH=E\oplus E^{\vee}$ corresponds to $\caC^{\infty}$ morphisms of vector bundles $\iota_2\colon E^{\vee}\ra\caH$ and $\pi_1\colon\caH\ra E$ such that
\begin{equation} \label{eq:direct-sum-projections}
\pi_1\circ\iota_1=\id_E,\quad\pi_2\circ\iota_2=\id_{E^{\vee}}\quad\text{and}\quad\iota_1\circ\pi_1+\iota_2\circ\pi_2=\id_{\caH}.
\end{equation}

Being the vector bundle associated to a local system, $\caH$ comes equipped with the Gauß-Manin connection
$$\nGM\colon\caA\left(\caH\right)\ra\caA^1\left(\caH\right)=\caA^{1,0}\left(\caH\right)\oplus\caA^{0,1}\left(\caH\right),$$
induced by the de Rham differential $d\colon\caC^{\infty}\left(B\right)\ra\caA^1_B=\caA^{1,0}_B\oplus\caA^{0,1}_B$ on $B$. This connection is holomorphic, i.e. the $\left(0,1\right)$-part is $\overline{\partial}_{\caH}\colon\caA\left(\caH\right)\ra\caA^{0,1}\left(\caH\right)$. Hence holomorphic sections of $\caH$ are mapped to holomorphic sections of $\caH\otimes\Omega_B^1$. Thus we will often also refer to the restriction $\nGM\colon\caH\ra\caH\otimes\Omega_B^1$ as the Gauß-Manin connection of the PVHS.

The Higgs field of the PVHS is
\begin{equation} \label{eq:df-Higgs}
\theta:=\pi_2\circ\nGM\circ\iota_1\colon E\ra E^{\vee}\otimes\Omega_B^1,
\end{equation}
which is $\caO_B$-linear, i.e. a (holomorphic) morphism of vector bundles. Moreover $\theta$ is symmetric, in the sense that $\theta^{\vee}=\theta\otimes\id_{T_B}\colon E\otimes T_B\ra E^{\vee}$.

\begin{definition} \label{df:Subbundles}
The Hodge bundle carries the following interesting subsheaves:
\begin{enumerate}
\item The \emph{flat unitary local system} associated to the PVHS \eqref{eq:PVHS} is
$$\mU:=\ker\left(\nGM\circ\iota_1\right)\subseteq \ker\nGM=\mV_{\mC}:=\mV_{\mZ}\otimes_{\mZ}\mC.$$
\item The associated holomorphic vector bundle $\caU:=\mU\otimes_{\mC}\caO_B\subseteq E$ is the \emph{flat unitary subbundle} of $E$.
\item The \emph{kernel subsheaf} is $\caK:=\ker\theta$, the kernel of the Higgs field.
\end{enumerate}
\end{definition}

Note that by construction, it holds $\caU\subseteq\caK$.

\begin{remark} \label{rmk:subbundles}
Some remark about the use of \emph{subbundle} and \emph{subsheaf}.
\begin{enumerate}
\item The flat subbundle $\caU$ is actually a vector subbundle, namely the fibres of $\caU$ have constant rank and inject in the fibres of $E$, hence the quotient $E/\caU$ is also a vector bundle.
\item On the contrary the \emph{sheaves} $\caK$ and $E/\caK$ are a priori only torsion-free (as subsheaves of the locally free sheaves $E$ and $E^{\vee}\otimes\Omega_B^1$ respectively). When $\dim B=1$, both $\caK$ and $E/\caK$ are locally free, hence $\caK\subset E$ is actually also a subbundle.
\end{enumerate}
\end{remark}

By definition, $\caK$ is the kernel of a morphism of locally free sheaves $\theta$, hence its fibre at a general point $b\in B$ is
$$\caK_b:=\caK\otimes\mC\left(b\right)=\ker\left(\theta_b\colon E_b\longra E_b^{\vee}\otimes\Omega_{B,b}^1\right).$$
In particular the rank of $\caK$ can be computed by studying the $\mC$-linear map $\theta_b$ at one general point $b\in B$.

On the other hand, the Gauß-Manin connection is not $\caO_B$-linear, hence the stalk of $\mU$ (i.e. the fibre of $\caU$) at any point $b\in B$ depends \emph{a priori} on the restriction of $E\subseteq\caH$ to an open neighbourhood of $b$.

The main result of this section is to show that, if $\dim B=1$ there is locally a family of $\caC^{\infty}$ symmetric morphisms of vector bundles $\eta^{\left(k\right)}\colon E\ra E^{\vee}$ whose kernels define $\caU$ at a general point $b\in B$.

To do so, we need to consider various connections induced by $\nGM$:

\begin{definition} \label{df:connections}
Denote $\nE$ and $\nEd$ denote the connections on $E$ and $E^{\vee}$ induced by $\nGM$ and the decomposition $\caH=E\oplus E^{\vee}$, i.e.
\begin{equation} \label{eq:nE}
\nE=\pi_1\circ\nGM\circ\iota_1\colon\caA\left(E\right)\ra\caA^1\left(E\right)
\end{equation}
and
\begin{equation} \label{eq:nEdual}
\nEd=\pi_2\circ\nGM\circ\iota_2\colon\caA\left(E^{\vee}\right)\ra\caA^1\left(E^{\vee}\right).
\end{equation}
Let also $T:=\Hom_{\caO_B}^s\left(E,E^{\vee}\right)=\Sym^2E^{\vee}$ and $\nH\colon\caA\left(T\right)\ra\caA^1\left(T\right)$ the connection defined by
\begin{equation} \label{eq:nH-prop}
\left(\nHapp{X}\eta\right)\left(\alpha\right)=\nEdapp{X}\left(\eta\left(\alpha\right)\right)-\eta\left(\nEapp{X}\alpha\right)\in\Gamma\left(\caA\left(E^{\vee}\right)\right)
\end{equation}
for any $\alpha\in\Gamma\left(\caA\left(E\right)\right), \eta\in\Gamma\left(\caA\left(T\right)\right)$ and $X\in\caA\left(T_{B,\mC}\right)$.
\end{definition}

Note that, since $\pi_1$ and $\iota_2$ are just $\caC^{\infty}$ morphisms of vector bundles, $\nE,\nEd$ and $\nH$ are not holomorphic connections. So in particular, even if $\alpha\in\Gamma\left(E\right)$ and $X\in\Gamma\left(T_B\right)$ are holomorphic, the section $\nEapp{X}\alpha\in\caA\left(E\right)$ is not necessarily holomorphic.

The formula \eqref{eq:nH-prop} defines $\nHapp{X}\eta$ as a morphism $E\ra E^{\vee}$, and it is straightforward to check that it is symmetric if $\eta$ is.

\begin{remark} \label{rmk:nE-vs-nGM}
By definition, a holomorphic section $\alpha\in\Gamma\left(E\right)$ is a section of $\caK$ if and only if $\nGM\alpha\in\Gamma\left(E\otimes\Omega_B^1\right)\subseteq\Gamma\left(\caH\otimes\Omega_B^1\right)$. This means that
\begin{equation} \label{eq:nE-vs-nGM}
\nE\alpha=\nGM\alpha\quad\forall\,\alpha\in\Gamma\left(\caK\right).
\end{equation}
In particular $\nE\alpha$ is holomorphic if $\alpha\in\Gamma\left(\caK\right)$.

Moreover, if $\caK\subseteq E$ is locally free (e.g. if $\dim B=1$), the formula \eqref{eq:nE-vs-nGM} holds also for smooth sections $\alpha\in\Gamma\left(\caA\left(\caK\right)\right)$.
\end{remark}

Assume from now on that $B\subseteq\mC$ is an open disk with coordinate $t$. Thus the tangent bundle $T_B$ is globally generated by $\abl{t}$, and contraction with $\abl{t}$ gives isomorphisms $\iota_{\abl{t}}\colon\Omega_B^1\cong\caO_B$ and $\iota_{\abl{t}}\colon\caA^1_B\cong\caA^0_B$. For any connection $\nabla$ we denote by $\nabla_t=\nabla_{\abl{t}}$ the corresponding derivation in the direction of $\abl{t}$.

We show now a characterization of the sections of $\caU$ using the Higgs field and the iterations of $\nGMapp{t}$. This is the first step towards our punctual characterization of the sections of $\caU$.

\begin{theorem} \label{thm:U-with-nablas}
Let $\alpha\in\Gamma\left(E\right)$. Then it holds
\begin{equation} \label{eq:U-with-nablas}
\alpha\in\Gamma\left(\caU\right)\Longlra\alpha,\nGMapp{t}\alpha,\left(\nGMapp{t}\right)^2\alpha,\ldots,\left(\nGMapp{t}\right)^{g-1}\alpha\in\Gamma\left(\caK\right).
\end{equation}
or equivalently
\begin{equation} \label{eq:U-with-theta}
\alpha\in\Gamma\left(\caU\right)\Longlra\theta\left(\alpha\right)=\theta\left(\nGMapp{t}\alpha\right)=\theta\left(\left(\nGMapp{t}\right)^2\alpha\right)=\ldots=\theta\left(\left(\nGMapp{t}\right)^{g-1}\alpha\right)=0.
\end{equation}
\end{theorem}
\begin{proof}
One direction is clear: if $\sigma_1,\ldots,\sigma_r\in\Gamma\left(\mU\right)$ is a basis of flat sections and $\alpha=\sum_{i=1}^r\alpha_i\sigma_i\in\Gamma\left(\caU\right)\subseteq\Gamma\left(\caK\right)$ for holomorphic functions $\alpha_1,\ldots,\alpha_r$, then $\left(\nGMapp{t}\right)^n\alpha=\sum_{i=1}^r\frac{\partial^n\alpha_i}{\partial t^n}\sigma_i\in\Gamma\left(\caU\right)\subseteq\Gamma\left(\caK\right)$ or every $n\geq 0$.

For the other direction, for any $n\geq 0$ let $\caK'_n\subseteq E$ be the smallest subsheaf generated by
$$\caU,\alpha,\nGMapp{t}\alpha,\ldots,\left(\nGMapp{t}\right)^n\alpha.$$
By definition it holds $\caU\subseteq\caK'_0\subseteq\caK'_1\subseteq\ldots\subseteq\caK'_n\subseteq\ldots\subseteq E$. Denote by $\caK'=\bigcup_{n\geq 0}\caK'_n\subseteq E$ the smallest subsheaf containing all the $\left(\nGMapp{t}\right)^i\alpha$, $i\geq 0$.

Note that if $\caK'_n=\caK'_{n+1}$, then the chain stabilizes, i.e. $\caK'=\caK'_m=\caK'_n$ for every $m\geq n$. Indeed, suppose that $\left(\nGMapp{t}\right)^{n+1}\alpha\in\Gamma\left(\caK'_n\right)$, so that $\caK'_{n+1}=\caK'_n$. Then there exist $\sigma\in\Gamma\left(\caU\right)$ and $\beta_1,\ldots,\beta_n\in\caO_B\left(B\right)$ such that
$$\left(\nGMapp{t}\right)^{n+1}\alpha=\sigma+\sum_{i=0}^n\beta_i\left(\nGMapp{t}\right)^i\alpha.$$
Thus applying $\nGMapp{t}$ again it holds
\begin{align*}
\left(\nGMapp{t}\right)^{n+2}\alpha&=\nGMapp{t}\sigma+\sum_{i=0}^n\ablt{\beta_i}{t}\left(\nGMapp{t}\right)^i\alpha+\sum_{i=0}^n\beta_i\left(\nGMapp{t}\right)^{i+1}\alpha\\
&=\nGMapp{t}\sigma+\ablt{\beta_0}{t}\alpha+\sum_{i=1}^n\left(\ablt{\beta_i}{t}+\beta_{i-1}\right)\left(\nGMapp{t}\right)^i\alpha+\beta_n\left(\nGMapp{t}\right)^{n+1}\alpha\in\Gamma\left(\caK'_n\right),
\end{align*}
and analogously $\left(\nGMapp{t}\right)^m\alpha\in\Gamma\left(\caK'_n\right)$ for any $m\geq n$. Since at each step of the ascending chain before the stabilization the rank raises by $1$ and $\rk E=g$, it is clear that $\caK'_{g-1}=\caK'$. 

Now by construction it holds $\nGMapp{t}\caK'\subseteq\caK'$, and also $\caK'\subseteq\caK$ by the hypothesis $\left(\nGMapp{t}\right)^i\alpha\in\Gamma\left(\caK\right)$ for every $i=0,1,\ldots,g-1$. We can then repeat the argument in \cite[Lemma 3.1]{GPT19} with $\caK'$ instead of $\caK$ and show that $\caK'=\caU$, hence in particular $\alpha\in\Gamma\left(\caU\right)$.
\end{proof}

Note that in \eqref{eq:U-with-nablas} we could have written $\nEapp{t}$ instead of $\nGMapp{t}$ because of Remark \ref{rmk:nE-vs-nGM}.

The vanishing conditions in \eqref{eq:U-with-theta} are still not defined fibrewise. In order to achieve this, we introduce the following morphisms of vector bundles.

\begin{definition} \label{df:etas}
For any $k\in\mN$ let $\eta^{\left(k\right)}\colon E\ra E^{\vee}$ be the smooth symmetric morphisms (i.e. sections of $\caA\left(T\right)$) defined recursively by
\begin{itemize}
\item $\eta^{\left(1\right)}:=\contr{\abl{t}}\theta\colon E\ra E^{\vee}$, and
\item $\eta^{\left(k\right)}:=\nHapp{t}\left(\eta^{\left(k-1\right)}\right)=\left(\nHapp{t}\right)^{k-1}\left(\eta^{\left(1\right)}\right)$ for $k\geq 2$.
\end{itemize}
\end{definition}

\begin{remark} \label{rmk:etas-holomorphic}
Note that $\eta^{\left(1\right)}$ is holomorphic, but the subsequent $\eta^{\left(k\right)}$ for $k>1$ are in general only smooth sections of $T$.

One can interpret them thus as morphisms of sheaves $\eta^{\left(k\right)}\colon E\ra\caC^{\infty}\left(E^{\vee}\right)$, where as usual a holomorphic vector bundle is identified with its sheaf of holomorphic sections.
\end{remark}

\begin{definition} \label{df:thetas-Kr}
For $r\geq 1$ we define
\begin{enumerate}
\item the morphism of sheaves $\theta^{\left(r\right)}=\left(\eta^{\left(1\right)},\ldots,\eta^{\left(r\right)}\right)\colon E\ra\caC^{\infty}\left(E^{\vee}\right)^{\oplus r}$
\item the subshef $\caK^{\left(r\right)}:=\ker\theta^{\left(r\right)}\subseteq E$.
\end{enumerate}
\end{definition}

\begin{proposition} \label{prop:thetas-Kr}
For every $r\geq 1$, $\caK^{\left(r\right)}\subseteq E$ is a holomorphic vector subbundle, i.e. it is a locally free $\caO_B$-module and the quotient $E/\caK^{\left(r\right)}$ is also locally free.
\end{proposition}
\begin{proof}
We proceed by inducion on $r$. The case $r=1$ is the statement of Remark \ref{rmk:subbundles}.(2).

Now suppose that the claim is true for some $r\geq 1$ and we will show that it is also true for $r+1$. Note that by definition, $\caK^{\left(r+1\right)}$ coincides with the kernel of the restriction
\begin{equation} \label{eq:proof-thetas-Kr-restriction}
{\eta^{\left(r+1\right)}}_{\mid\caK^{\left(r\right)}}\colon\caK^{\left(r\right)}\ra E^{\vee}.
\end{equation}
Thus it is enough to show that \eqref{eq:proof-thetas-Kr-restriction} is a holomorphic morphism of vector bundles. Indeed this would imply that both $\caK^{\left(r+1\right)}\subseteq\caK^{\left(r\right)}$ and $\caK^{\left(r\right)}/\caK^{\left(r+1\right)}\subseteq E^{\vee}$ are subsheaves of locally free sheaves over a curve, hence locally free themselves.

We thus need to show that, if $\alpha$ is a holomorphic section of $\caK^{\left(r\right)}$, then $\eta^{\left(r+1\right)}\left(\alpha\right)$ is a holomorphic section of $E^{\vee}$. By definition of $\eta^{\left(r+1\right)}$ and $\caK^{\left(r\right)}$ it holds
$$\eta^{\left(r+1\right)}\left(\alpha\right)=\nEdapp{t}\left(\eta^{\left(r\right)}\left(\alpha\right)\right)-\eta^{\left(r\right)}\left(\nEapp{t}\left(\alpha\right)\right)=-\eta^{\left(r\right)}\left(\nGMapp{t}\left(\alpha\right)\right).$$
Note that in the last equality we have used $\nEapp{t}\left(\alpha\right)=\nGMapp{t}\left(\alpha\right)$, which holds because $\nE=\nGM$ on $\caK=\caK^{\left(1\right)}\supseteq\caK^{\left(r\right)}$.

It remains only to show that $\eta^{\left(r\right)}\left(\nGMapp{t}\left(\alpha\right)\right)$ is holomorphic. Since $\nGMapp{t}\left(\alpha\right)$ is holomorphic and $\eta^{\left(r\right)}$ is holomorphic on $\caK^{\left(r-1\right)}$ (by induction), it is enough to show that $\nGMapp{t}\alpha$ is indeed a section of $\caK^{\left(r-1\right)}$. But this follows easily from
$$\eta^{\left(s\right)}\left(\nEapp{t}\alpha\right)=\nEdapp{t}\left(\eta^{\left(s\right)}\left(\alpha\right)\right)-\eta^{\left(s+1\right)}\left(\alpha\right)=0$$
for every $s=1,\ldots,r-1$.
\end{proof}

\begin{theorem} \label{thm:U-with-etas}
Let $\alpha\in\Gamma\left(E\right)$. Then it holds $\caU=\caK^{\left(g\right)}$, i.e.
\begin{equation} \label{eq:U-with-etas}
\alpha\in\Gamma\left(\caU\right)\Longlra\eta^{\left(1\right)}\left(\alpha\right)=\eta^{\left(2\right)}\left(\alpha\right)=\ldots=\eta^{\left(g\right)}\left(\alpha\right)=0.
\end{equation}
In particular we have
\begin{equation} \label{eq:U-inters-kernels}
\caU_b\subseteq\bigcap_{k=1}^{g}\ker\eta^{\left(k\right)}_b\subseteq E_b
\end{equation}
with equality for $b$ in a dense Zariski-open subset of $B$.
\end{theorem}
\begin{proof}
By Theorem \ref{thm:U-with-nablas}, after contracting with $\abl{t}$ it is enough to show that the conditions
\begin{equation} \label{eq:cond-with-theta}
\eta^{\left(1\right)}\left(\alpha\right)=\eta^{\left(1\right)}\left(\nGMapp{t}\alpha\right)=\eta^{\left(1\right)}\left(\left(\nGMapp{t}\right)^2\alpha\right)=\ldots=\eta^{\left(1\right)}\left(\left(\nGMapp{t}\right)^n\alpha\right)=0.
\end{equation}
and
\begin{equation} \label{eq:cond-with-etas}
\eta^{\left(1\right)}\left(\alpha\right)=\eta^{\left(2\right)}\left(\alpha\right)=\ldots=\eta^{\left(n\right)}\left(\alpha\right)=0.
\end{equation}
are equivalent for any $n\in\mN$. We will actually show that both \eqref{eq:cond-with-theta} and \eqref{eq:cond-with-etas} are equivalent to the condition
\begin{equation} \label{eq:mixed-conditions}
\eta^{\left(i\right)}\left(\left(\nGMapp{t}\right)^j\alpha\right)=0\quad\forall\,i+j\leq n.
\end{equation}
It is obvious that \eqref{eq:mixed-conditions} implies \eqref{eq:cond-with-theta} (setting $i=1$) and \eqref{eq:cond-with-etas} (setting $j=0$).

Suppose now that \eqref{eq:cond-with-theta} holds, so that \eqref{eq:mixed-conditions} holds for $i=1$ and any $j$. In particular we have $\left(\nGMapp{t}\right)^j\alpha\in\Gamma\left(\caK\right)$ for every $j$, and therefore $\nEapp{t}\left(\left(\nGMapp{t}\right)^j\alpha\right)=\left(\nGMapp{t}\right)^{j+1}\alpha$ by \eqref{eq:nE-vs-nGM}.

We proceed by induction on $i$, assuming that \eqref{eq:mixed-conditions} holds for a fixed $i\geq 1$ and any $j\leq n-i$. Then the identity \eqref{eq:nH-prop} implies
\begin{align*}
\eta^{\left(i+1\right)}\left(\left(\nGMapp{t}\right)^j\alpha\right)&=\left(\nHapp{t}\eta^{\left(i\right)}\right)\left(\left(\nGMapp{t}\right)^j\alpha\right)\\
&=\nEdapp{t}\left(\eta^{\left(i\right)}\left(\left(\nGMapp{t}\right)^j\alpha\right)\right)-\eta^{\left(i\right)}\left(\nEapp{t}\left(\left(\nGMapp{t}\right)^j\alpha\right)\right)\\
&=\nEdapp{t}\left(\eta^{\left(i\right)}\left(\left(\nGMapp{t}\right)^j\alpha\right)\right)-\eta^{\left(i\right)}\left(\left(\nGMapp{t}\right)^{j+1}\alpha\right)=0-0=0
\end{align*}
for any $j\leq n-\left(i+1\right)$, i.e. \eqref{eq:mixed-conditions} holds also for $i+1$.

We finally show, by induction on $j$, that \eqref{eq:cond-with-etas} implies \eqref{eq:mixed-conditions}. Indeed \eqref{eq:cond-with-etas} is the case $j=0$ of \eqref{eq:mixed-conditions}. Assuming now that \eqref{eq:mixed-conditions} holds for a fixed $j$ and every $i\leq n-j$, the same computations as above show
$$\eta^{\left(i\right)}\left(\left(\nGMapp{t}\right)^{j+1}\alpha\right)=\nEdapp{t}\left(\eta^{\left(i\right)}\left(\left(\nGMapp{t}\right)^j\alpha\right)\right)-\eta^{\left(i+1\right)}\left(\left(\nGMapp{t}\right)^j\alpha\right)=0-0=0$$
for any $i\leq n-\left(j+1\right)$. This means that \eqref{eq:mixed-conditions} holds for $j+1$, which concludes the proof of \eqref{eq:U-with-etas}.

The inclusions \eqref{eq:U-inters-kernels} and the equality on the first inclusion for general $b$ follow from Proposition \ref{prop:thetas-Kr}.
\end{proof}

We close this section putting our constructions in the context of totally geodesic submanifolds of the Siegel upper-half space $\mH_g$ (and of the moduli space $\caA_ g$). Indeed by chosing a frame of $\mV_{\mZ}$ we obtain (a lifting of) the period map $\gamma\colon B\ra\mH_g$, with the property that $T=\Hom^s\left(E,E^{\vee}\right)=\gamma^*T_{\mH_g}$ and the Higgs field can be interpreted as the differential of $\gamma$
$$d\gamma\colon T_B\longra\Hom^s\left(E,E^{\vee}\right).$$
Moreover, the conection $\nH$ coincides with the pullback of the metric Levi-Civita connection of the Siegel metric in $\mH_g$.

Under these identifications, and assuming that $\gamma$ is an embedding, we have that $\eta^{\left(1\right)}=d\gamma\left(\abl{t}\right)$ is a non-vanishing tangent vector field along $B\subseteq\mH_g$. Moreover $\eta^{\left(2\right)}=\nHapp{t}\eta^{\left(1\right)}$ is the covariant derivative of $\eta^{\left(1\right)}$ along itself, and the $\eta^{\left(k\right)}$ are the subsequent derivatives.

Recall from the introduction that in some recent works Chen, Lu and Zuo but also many other authors relate the rank of $\caU$ with the property of $B\subseteq\mH_g$ being totally geodesic. Of particular interest for the present paper is the result by Ghigi, Pirola and the second author in \cite{GPT19} that if $B$ contains one (real) geodesic, then $\caU=\caK$. The proof is basically consequence of \cite[Lemma 3.1]{GPT19} which we apply here to prove Proposition \ref{thm:U-with-nablas} in a slighlty more general version.

\section{Families of curves and geometric VHS}

In this section we consider the tools of the previous section in the case of a geometric PVHS of weight one, i.e. which arises from a family of smooth projective curves. We study in detail the action of the second additional vector field $\eta^{\left(2\right)}$ in terms of the coordinate expressions of the relative $1$-forms. In particular, we define a cohomology class $\mu$ of the tangent bundle of each fibre, which acts almost like $\eta^{\left(2\right)}$ and also gives a second linear condition defining $\caU\subseteq\caK$. These classes $\mu$ depend on the second-order neighbourhood of the fibres and we call them \emph{second-order} Kodaira-Spencer class of the deformations.

To be precise, we consider a smooth family of projective curves $f\colon\caC\ra B$ over a smooth complex manifold $B$. We denote by $C_b=f^{-1}\left(b\right)$ the curves of the family, all of which have the same genus $g$. We have then the mutually dual short exact sequences of vector bundles on $\caC$
\begin{equation} \label{eq:relative-cotangent-family}
0\longra f^*\Omega_B^1\longra \Omega_{\caC}^1\longra\Omega_{\caC/B}^1\cong\omega_{\caC/B}\longra 0
\end{equation}
\begin{equation} \label{eq:relative-tangent-family}
0\longra T_{\caC/B}\longra T_{\caC}\longra f^*T_B\longra 0.
\end{equation}
The Hodge structures of the fibres $C_b$ form a {\em geometric} PVHS
$$0 \longra f_*\omega_{\caC/B} \longra R^1f_*\mZ_{\caC}\otimes_{\mZ}\caO_B \longra R^1f_*\caO_{\caC} \longra 0,$$
i.e. we have $\mV_{\mZ}=R^1f_*\mZ_{\caC}$, $E=f_*\omega_{\caC/B}$ and the polarization $Q$ is induced by the cup product on the stalks $H^1\left(C_b,\mZ\right)$.

The Higgs field is the connecting homomorphism
$$\theta\colon f_*\omega_{\caC/B}\ra R^1f_*\caO_{\caC}\otimes\Omega^1_B$$
obtained by pushing forward the short exact sequence \eqref{eq:relative-cotangent-family}.

Moreover at any point $b\in B$, $\theta_b\colon H^0(C_b,\Omega_{C_b}^1)\ra H^1\left(C_b,\caO_{C_b}\right)\otimes T_{B,b}^{\vee}$ is given by cup- and interior product with the Kodaira-Spencer map $KS_b\colon T_{B,b}\ra H^1\left(C_b,T_{C_b}\right)$. More explicitly, for every $v\in T_{B,b}$ the class $\xi_v:=KS_b\left(v\right)\in H^1\left(T_{C_b}\right)$ corresponds to the first order deformation of $C_b$ inside $\caC$ in the direction of $v$, and it holds
\begin{equation} \label{eq:theta-KS}
\theta_b\left(\alpha\right)\left(v\right)=\xi_v\cdot \alpha\in H^1\left(C_b,\caO_{C_b}\right).
\end{equation}
In particular, for general $b\in B$ it holds
\begin{equation} \label{eq:inf-K}
\caK\otimes\mC\left(b\right)=\bigcap_{v\in T_{B,b}\setminus\left\{0\right\}}\ker\left(\xi_v\cdot \colon H^0\left(C_b,\Omega_{C_b}^1\right)\ra H^1\left(C_b,\caO_{C_b}\right)\right).
\end{equation}

The interpretation of the Higgs field as the connecting homomorphism of the push-forward of \eqref{eq:relative-cotangent-family} gives the following interpretation of the local sections of $\caK$: if $V\subseteq B$ is an open disk, then 
\begin{equation} \label{eq:local-sections-K}
\Gamma\left(V,\caK\right)=\im\left(\Gamma\left(f^{-1}\left(V\right),\Omega_{\caC}^1\right)\longra\Gamma\left(f^{-1}\left(V\right),\omega_{\caC/B}\right)\right),
\end{equation}
namely sections of $\caK$ over $V$ correspond to families of holomorphic $1$-forms on the fibres $C_b$ (for $b\in V$) arising as restrictions of a common holomorphic $1$-form on $f^{-1}\left(V\right)$. With this in mind, Pirola and the second author proved in \cite{PT} that 
\begin{equation} \label{eq:local-flat-sections-U}
\Gamma\left(V,\mU\right)=\im\left(\Gamma\left(f^{-1}\left(V\right),\Omega_{\caC,d}^1\right)\longra\Gamma\left(f^{-1}\left(V\right),\omega_{\caC/B}\right)\right),
\end{equation}
where $\Omega_{\caC,d}^1:=\ker\left(d\colon\Omega_{\caC}^1\ra\Omega_{\caC}^2\right)$ is the sheaf of closed holomorphic $1$-forms on $\caC$. In other words, the flat sections of $E$ correspond to families of $1$-forms on the fibres arising as restriction of a common \emph{closed} holomorphic $1$-form on $f^{-1}\left(V\right)$.

The closedness condition still cannot be checked on a given fibre $C_b$ fibre, but only on an open neighbourhod. In order to characterize $\caU_b$ just from infinitesimal data of $C_b$ inside $\caC$ we need to understand the homomorphisms $\eta^{\left(k\right)}_b\colon H^0\left(\omega_{C_b}\right)\ra H^1\left(\caO_{C_b}\right)$ induced by the $\eta^{\left(k\right)}$ from Definition \ref{df:etas} in this geometric case. In this paper we will focus in the case $k=2$, i.e. the first additional linear condition.

Although the $\eta^{\left(k\right)}$ were defined for $\dim B=1$ with a global coordinate $t$, we will make some preliminar computations in a more general setting, where $B\subseteq\mC^r$ is an open ball with coordinates $\left(t_1,\ldots,t_r\right)$.

Since the family is assumed to be smooth, around any point $P\in\caC$ we can find a function $x$ such that $\left(x,t_1,\ldots,t_r\right)$ is a coordinate system around $P$. This means, the map $f$ is given by $\left(x,t_1,\ldots,t_r\right)\mapsto\left(t_1,\ldots,t_r\right)$ and $x$ restricts to a local coordinate on the fibres $C_b$. In this setting $f^*\Omega_B^1$ is (globally) generated by $\left\{dt_1,\ldots,dt_r\right\}$ and $\Omega_{\caC}^1$ resp. $\omega_{\caC/B}$ are locally generated by $\left\{dx,dt_1,\ldots,dt_n\right\}$ resp. $\left\{dx\right\}$. If $\left(x',t_1,\ldots,t_n\right)$ is another coordinate system, it holds
\begin{equation} \label{eq:change-dx-Omega}
dx'=\ablt{x'}{x}dx+\sum_{i=1}^n\ablt{x'}{t_i}dt_i
\end{equation}
in $\Omega_{\caC}^1$, while in $\omega_{\caC/B}$ we only have
\begin{equation} \label{eq:change-dx-omega}
dx'=\ablt{x'}{x}dx.
\end{equation}

A holomorphic section $\alpha\in\Gamma\left(E\right)=\Gamma\left(\omega_{\caC/B}\right)$ can be interpreted as a holomorphic section (on $\caC$) of $\omega_{\caC/B}$. According to the previous discussion, $\alpha$ can be locally written as $a\left(x,t_1,\ldots,t_r\right)dx$ with a holomorphic coefficient function $a$. If $a'\left(x',t_1,\ldots,t_r\right)dx'$ is another local expression of $\alpha$, it holds then $a=a'\ablt{x'}{x}$. The same holds for a smooth section $\alpha\in\Gamma\left(\caA\left(E\right)\right)$, with the only additional condition that the local smooth function $a$ should be holomorphic on $x$, i.e. $\ablt{a}{\overline{x}}=0$.

From the exact sequence \eqref{eq:relative-cotangent-family}, given any smooth section $\alpha\in\Gamma\left(\caA\left(E\right)\right)$ we can find a \emph{smooth} lifting $\widetilde{\alpha}\in\Gamma\left(\caA\left(\Omega_{\caC}^1\right)\right)$, with local expression
\begin{equation} \label{eq:lifting-alpha}
\widetilde{\alpha}=a\left(x,t_1,\ldots,t_r\right)dx+\sum_{i=1}^rc_i\left(x,t_1,\ldots,t_r\right)dt_i.
\end{equation}
Even if $\alpha$ is holomorphic, so that the function $a$ is holomorphic by assumption, the additional coefficient functions $c_1,\ldots,c_r$ are a priori only smooth functions. Actually, they can be chosen to be holomorphic if and only if $\alpha$ is a (holomorphic) section of $\caK$, but not otherwise.

More generally, any smooth section $\alpha$ of $\caH=R^1f_*\mZ_{\caC}\otimes_{\mZ}\caO_B$ can be represented by a $1$-form $\widetilde{\alpha}\in\caA^1\left(\caC\right)$ with the additional condition that $d\widetilde{\alpha}_{\mid C_b}=0$ for any $b\in B$, so that
$$\alpha\left(b\right)=\left[\widetilde{\alpha}_{\mid C_b}\right]\in H^1\left(C_b,\mC\right).$$
In the case of sections of $E$ the closedness condition $d\widetilde{\alpha}_{\mid C_b}=0$ is automatically satisfied.

In order to obtain a precise description of the action of $\eta^{\left(2\right)}$, we first need some explicit description of the connection $\nH$, for which in turn we need some explicit expression for the Gauß-Manin connection. This is given by the Cartan-Lie formula (see \cite[Proposition 9.14]{Voi1})
\begin{equation} \label{eq:Cartan-Lie}
\left(\nGMapp{X}\alpha\right)_b=\left[\left(\contr{\chi}d\widetilde{\alpha}\right)_{\mid C_b}\right]\in H^1\left(C_b,\mC\right),
\end{equation}
where
\begin{enumerate}
\item $X\in\Gamma\left(B,\caA\left(T_{B,\mC}\right)\right)$ is any smooth (complex) vector field on $B$,
\item $\chi$ is a smooth vector field on $\caC$ that descends to $X$, i.e. a $\caC^{\infty}$ lifting of $X$ under the short exact sequence
\begin{equation} \label{eq:normal-bundle-complex}
0\longra T_{\caC/B,\mC} \longra T_{\caC,\mC} \stackrel{df}{\longra} f^*T_{B,\mC}\longra 0
\end{equation}
\item $\alpha$ is any smooth section of $\caH$, represented by a $1$-form $\widetilde{\alpha}\in\caA^1\left(S\right)$ with $d\widetilde{\alpha}_{\mid C_b}=0$ for any $b\in B$, so that $\alpha\left(b\right)=\left[\widetilde{\alpha}_{\mid C_b}\right]\in H^1\left(C_b,\mC\right)$.
\item In particular, if $\alpha$ is a section of $E$, we can assume $\widetilde{\alpha}\in\caA^{1,0}\left(\caC\right)$.
\end{enumerate}

In this setting, we look now for an explicit description of $\nH$. It is enough for our purposes to understand the action of $\nH$ on the homomorphisms $E\ra E^{\vee}$ arising from the Higgs field $\theta$. For the sake of simplicity, we introduce the following notations.

\begin{definition} \label{df:xi-i-eta-ij}
For each $i=1,\ldots,r$ let:
\begin{enumerate}
\item $\xi_i:=\iota_{\abl{t_i}}\theta\colon E\ra E^{\vee}$, which are holomorphic homomorphisms of vector bundles, and
\item $\chi_i\in\caA\left(T_{\caC}\right)$, a smooth vector field on $\caC$ descending to $\abl{t_i}$. In an open subset $V$ of $\caC$ with coordinates $\left(x,t_1,\ldots,t_r\right)$ it can be written as
\begin{equation} \label{eq:local-expression-chi-i}
\chi_i=Z_i\abl{x}+\abl{t_i}
\end{equation}
for a certain smooth function $Z_i\colon V\ra\mC$.
\end{enumerate}
For each $i,j=1,\ldots,r$ set $\eta_{ij}:=\nHapp{\abl{t_j}}\xi_i$.
\end{definition}

For any holomorphic section $\alpha\in\Gamma\left(E\right)$ we will find a quite explicit coordinate description of
$$\eta_{ij}\left(\alpha\right)=\left(\nHapp{\abl{t_j}}\xi_i\right)\left(\alpha\right)=\nEdapp{\abl{t_j}}\left(\xi_i\left(\alpha\right)\right)-\xi_i\left(\nEapp{\abl{t_j}}\alpha\right).$$
We present now a step by step computation, since we need to define some additional functions at some intermediate steps, and summarize the result into a formal statement at the end.

Let $\widetilde{\alpha}\in\Gamma\left(\caC^{\infty}\left(\Omega_{\caC}^1\right)\right)=\Gamma\left(\caA^{1,0}\left(\caC\right)\right)$ be a representative of $\alpha$ with local expression $adx+\sum_{k=1}^nc_idt_i$. Chosen the vector fields $\chi_k$, there is an obvious choice $c_k=-aZ_k$ for every $k=1,\ldots,r$, which we will use later on. But for the sake of simplicity of notation, we will keep $c_k$ for the moment.

By \eqref{eq:Cartan-Lie}, for any $b\in B$, $\left(\nGMapp{\xi_i}\alpha\right)_b$ is represented by $\left[\contr{\chi_i}d\widetilde{\alpha}_{\mid C_b}\right]\in H^1\left(C_b,\mC\right)$. 

We compute first
\begin{align*} 
\beta_i&:=\contr{\chi_i}d\widetilde{\alpha} \label{eq:beta-i}\\
&=\contr{\chi_i}\left(-\sum_k\left(\ablt{a}{t_k}-\ablt{c_k}{x}\right)dx\wedge dt_k+\sum_{k,l}\ablt{c_k}{t_l}dt_l\wedge dt_k+\sum_k\ablt{c_k}{\overline{x}}d\overline{x}\wedge dt_k+\sum_{k,l}\ablt{c_k}{\overline{t}_l}d\overline{t}_l\wedge dt_k\right)\nonumber \\
&=\left(\ablt{a}{t_i}-\ablt{c_i}{x}\right)dx-\sum_k\left(\ablt{a}{t_k}-\ablt{c_k}{x}\right)Z_idt_k+\sum_k\left(\ablt{c_k}{t_i}-\ablt{c_i}{t_k}\right)dt_k-\ablt{c_i}{\overline{x}}d\overline{x}-\sum_k\ablt{c_i}{\overline{t}_k}d\overline{t}_k \nonumber 
\end{align*}

We now need to split $\nGMapp{\xi_i}\alpha$ into its parts of type $\left(1,0\right)$ and $\left(0,1\right)$, namely $\nEapp{\xi_i}\alpha$ and $\xi_i\left(\alpha\right)$ (or more precisely $\iota_1\left(\nEapp{\xi_i}\alpha\right)$ and $\iota_2\left(\xi_i\left(\alpha\right)\right)$), for which we need to take the harmonic representatives of the $\left(\nGMapp{\xi_i}\alpha\right)_b$.

\begin{definition} \label{df:phis}
For each $i=1,\ldots,r$, let $\varphi_i=\varphi_{i,\alpha}\colon\caC\ra\mC$ be a smooth function such that the restrictions $\left(\widetilde{\beta_i}\right)_b:=\left(\beta_i+d\varphi_i\right)_{\mid C_b}$ are harmonic for every $b\in B$.
\end{definition}

These functions $\varphi_1,\ldots,\varphi_r$ exist by general theory and are well defined up to constants depending on the $t_i's$. Then the components of type $\left(1,0\right)$ and $\left(0,1\right)$ of $\left(\widetilde{\beta_i}\right)_b$ are also harmonic and define the Hodge decomposition of $\left(\nGMapp{\xi_i}\alpha\right)_b$, i.e. $\nEapp{\xi_i}\alpha$ is represented by
\begin{equation}\label{eq:nE-alpha-chi-i}
\widetilde{\beta_i}^{1,0}=\beta_i^{1,0}+\partial\varphi_i=\left(\ablt{a}{t_i}-\ablt{c_i}{x}\right)dx-\sum_k\left(\ablt{a}{t_k}-\ablt{c_k}{x}\right)Z_idt_k+\sum_k\left(\ablt{c_k}{t_i}-\ablt{c_i}{t_k}\right)dt_k+\ablt{\varphi_i}{x}dx+\sum_k\ablt{\varphi_i}{t_k}dt_k,
\end{equation}
and
\begin{equation}\label{eq:xi-i-alpha}
\widetilde{\beta_i}^{0,1}=\beta_i^{0,1}+\overline{\partial}\varphi_i=-\ablt{c_i}{\overline{x}}d\overline{x}-\sum_k\ablt{c_i}{\overline{t}_k}d\overline{t}_k+\ablt{\varphi_i}{\overline{x}}d\overline{x}+\sum_k\ablt{\varphi_i}{\overline{t}_k}d\overline{t}_k
\end{equation}
represents $\xi_i\left(\alpha\right)$ as a section of $\caH$ by means of the splitting $\iota_2\colon E^{\vee}\hra\caH$.

\begin{remark} \label{rmk:KS-again}
Indeed, $\xi_i\left(\alpha\right)$ as a section of $E^{\vee}$ can also be directly represented by $\beta^{0,1}$, i.e.
\begin{equation} \label{eq:theta-i-alpha}
\left(\xi_i\left(\alpha\right)\right)_b=\left[-\ablt{c_i}{\overline{x}}d\overline{x}_{\mid C_b}\right]=\left[a\ablt{Z_i}{\overline{x}}d\overline{x}_{\mid C_b}\right]\in H^1_{\overline{\partial}}\left(\caO_{C_b}\right),
\end{equation}
where the last equality follows by taking $c_i=-aZ_i$, or by adding the global $\overline{\partial}$-exact $\left(0,1\right)$-form $\overline{\partial}\left(\contr{\chi_i}\widetilde{\alpha}\right)_{\mid C_b}$ whose local expression is precisely $\overline{\partial}\left(aZ_i+c_i\right)_{\mid C_b}$.

From the Dolbeaut construction of the connecting homomorphism in coholomology of the exact sequence \eqref{eq:relative-tangent-family} it follows that the Kodaira-Spencer class $KS_b\left(\abl{t_i}\right)\in H^1\left(T_{C_b}\right)$ is represented in Dolbeault cohomology by the $\left(0,1\right)$-form with values in $T_{C_b}$ whose local expression is
$$\ablt{Z_i}{\overline{x}}d\overline{x}\otimes\abl{x}_{\mid C_b}.$$
Thus equation \eqref{eq:theta-i-alpha} recovers the formula \eqref{eq:theta-KS}, i.e. shows that $\left(\xi_i\right)_b$ acts by cup- and interior product with $KS_b\left(\abl{t_i}\right)$.
\end{remark}

It remains now to apply \eqref{eq:Cartan-Lie} again, keeping just the terms in $d\overline{x}$ (since every other term vanishes when restricting to $C_b$ or when projecting to $H^1_{\overline{\partial}}\left(\caO_{C_b}\right)\cong H^{0,1}\left(C_b\right)=E^{\vee}_b$), to compute
\begin{align*}
\nEdapp{\abl{t_j}}\left(\xi_i\left(\alpha\right)\right)_b&=\left[\contr{\chi_j}d\widetilde{\beta}^{0,1}_{\mid C_b}\right]^{0,1}\\
&=\left[\left(-\frac{\partial^2c_i}{\partial\overline{x}\partial x}Z_j-\frac{\partial^2c_i}{\partial\overline{x}\partial t_j}+Z_j\frac{\partial^2\varphi_i}{\partial\overline{x}\partial x}+\frac{\partial^2\varphi_i}{\partial\overline{x}\partial t_j}\right)_{\mid C_b}d\overline{x}\right] \\
&=\left[\left(-\frac{\partial^2c_i}{\partial\overline{x}\partial x}Z_j-\frac{\partial^2c_i}{\partial\overline{x}\partial t_j}+\chi_j\left(\ablt{\varphi_i}{\overline{x}}\right)\right)_{\mid C_b}d\overline{x}\right]
\end{align*}
\begin{align*}
\xi_j\left(\nEapp{\abl{t_i}}\alpha\right)_b&=\left[\contr{\chi_j}d\widetilde{\beta}^{1,0}_{\mid C_b}\right]^{0,1}\\
&=\left[\left(\frac{\partial^2c_i}{\partial\overline{x}\partial x}Z_j+\ablt{}{\overline{x}}\left(\left(\ablt{a}{t_j}-\ablt{c_j}{x}\right)Z_i\right)-\ablt{}{\overline{x}}\left(\ablt{c_j}{t_i}-\ablt{c_i}{t_j}\right)-Z_j\frac{\partial^2\varphi_i}{\partial\overline{x}\partial x}-\frac{\partial^2\varphi_i}{\partial\overline{x}\partial t_j}\right)_{\mid C_b}d\overline{x}\right]\\
&=\left[\left(\frac{\partial^2c_i}{\partial\overline{x}\partial x}Z_j-\frac{\partial^2 c_j}{\partial\overline{x}\partial x}Z_i+\ablt{a}{t_j}\ablt{Z_i}{\overline{x}}-\ablt{c_j}{x}\ablt{Z_i}{\overline{x}}-\frac{\partial^2c_j}{\partial\overline{x}\partial t_i}+\frac{\partial^2c_i}{\partial\overline{x}\partial t_j}-\chi_j\left(\ablt{\varphi_i}{\overline{x}}\right)\right)_{\mid C_b}d\overline{x}\right]
\end{align*}
and finally (note the exchange of indices $i$ and $j$ in the second summand)
\begin{align*}
\left(\nHapp{\abl{t_j}}\xi_i\right)\left(\alpha\right)_b&=\nEdapp{\abl{t_j}}\left(\xi_i\left(\alpha\right)\right)_b-\xi_i\left(\nEapp{\abl{t_j}}\alpha\right)_b\\
&=\left[\left(-\frac{\partial^2c_i}{\partial\overline{x}\partial x}Z_j-\frac{\partial^2c_i}{\partial\overline{x}\partial t_j}+\chi_j\left(\ablt{\varphi_i}{\overline{x}}\right)\right.\right.\\
&\quad\quad\left.\left.-\frac{\partial^2c_j}{\partial\overline{x}\partial x}Z_i+\frac{\partial^2 c_i}{\partial\overline{x}\partial x}Z_j-\ablt{a}{t_i}\ablt{Z_j}{\overline{x}}+\ablt{c_i}{x}\ablt{Z_j}{\overline{x}}+\frac{\partial^2c_i}{\partial\overline{x}\partial t_j}-\frac{\partial^2c_j}{\partial\overline{x}\partial t_i}+\chi_i\left(\ablt{\varphi_j}{\overline{x}}\right)\right)_{\mid C_b}d\overline{x}\right]\\
&=\left[\left(-\frac{\partial^2c_j}{\partial\overline{x}\partial x}Z_i-\ablt{a}{t_i}\ablt{Z_j}{\overline{x}}+\ablt{c_i}{x}\ablt{Z_j}{\overline{x}}-\frac{\partial^2c_j}{\partial\overline{x}\partial t_i}+\chi_j\left(\ablt{\varphi_i}{\overline{x}}\right)+\chi_i\left(\ablt{\varphi_j}{\overline{x}}\right)\right)_{\mid C_b}d\overline{x}\right].
\end{align*}
If we assume $c_i=-aZ_i$ and $c_j=-aZ_j$, the formula becomes:
\begin{equation} \label{eq:eta-ij-1}
\left(\nHapp{\abl{t_j}}\xi_i\right)\left(\alpha\right)=\left[\left(a\left(\frac{\partial^2Z_j}{\partial\overline{x}\partial x}Z_i-\ablt{Z_i}{x}\ablt{Z_j}{\overline{x}}+\frac{\partial^2 Z_j}{\partial\overline{x}\partial t_i}\right)+\chi_j\left(\ablt{\varphi_i}{\overline{x}}\right)+\chi_i\left(\ablt{\varphi_j}{\overline{x}}\right)\right)_{\mid C_b}d\overline{x}\right].
\end{equation}
We can obtain a nicer formula if we change the representative in $H^1_{\overline{\partial}}\left(\caO_{C_b}\right)$ by adding the $\overline{\partial}$-exact form $\overline{\partial}\left(\chi_i\left(\varphi_j\right)+\chi_j\left(\varphi_i\right)\right)_{\mid C_b}$. In the chosen local coordinates we have
$$\overline{\partial}\left(\chi_i\left(\varphi_j\right)\right)_{\mid C_b}=\overline{\partial}\left(Z_i\ablt{\varphi_j}{x}+\ablt{\varphi_j}{t_i}\right)=\left(\chi_i\left(\ablt{\varphi_j}{\overline{x}}\right)+\ablt{Z_i}{\overline{x}}\ablt{\varphi_j}{x}\right)d\overline{x}.$$
Thus substracting $\overline{\partial}\left(\chi_i\left(\varphi_j\right)+\chi_j\left(\varphi_i\right)\right)_{\mid C_b}$ from \eqref{eq:eta-ij-1} we obtain the following

\begin{theorem} \label{thm:eta-ij}
If $\alpha\in\Gamma\left(E\right)$ is locally given by $a\left(x,t_1,\ldots,t_r\right)dx$, then for every $i,j=1,\ldots,r$
\begin{equation} \label{eq:eta-ij-2}
\left(\nHapp{\abl{t_j}}\xi_i\right)\left(\alpha\right)_b=\left[\left(a\left(\frac{\partial^2Z_j}{\partial\overline{x}\partial x}Z_i+\frac{\partial^2 Z_j}{\partial\overline{x}\partial t_i}-\ablt{Z_j}{\overline{x}}\ablt{Z_i}{x}\right)-\ablt{Z_i}{\overline{x}}\ablt{\varphi_j}{x}-\ablt{Z_j}{\overline{x}}\ablt{\varphi_i}{x}\right)_{\mid C_b}d\overline{x}\right],
\end{equation}
where the globally defined functions $\varphi_i,\varphi_j\colon\caC\ra\mC$ are as in Definition \ref{df:phis}.
\end{theorem}

We restrict ourselves now to the case where $\dim B=1$ and $B$ has a global coordinate $t$. In this case we have only one $\xi=\eta^{\left(1\right)}=\iota_{\abl{t}}\theta$ and $\eta^{\left(2\right)}=\nHapp{\abl{t}}\xi$, and we write $Z$ for the function such that the vector field lifting $\abl{t}$ can be written as $\chi=Z\abl{x}+\abl{t}$. In this setting Theorem \ref{thm:eta-ij} immediately gives:

\begin{corollary} \label{cor:eta}
If $\alpha\in\Gamma\left(E\right)$ is locally given by $a\left(x,t\right)dx$, then at every $b\in B$ it holds
\begin{equation} \label{eq:eta}
\left(\eta^{\left(2\right)}\right)\left(\alpha\right)_b=\left[\left(a\left(Z\frac{\partial^2Z}{\partial\overline{x}\partial x}+\frac{\partial^2Z}{\partial\overline{x}\partial t}-\ablt{Z}{\overline{x}}\ablt{Z}{x}\right)-2\ablt{Z}{\overline{x}}\ablt{\varphi}{x}\right)_{\mid C_b}d\overline{x}\right],
\end{equation}
where $\varphi\colon\caC\ra\mC$ is a smooth function satisfying the analogous condition as in Definition \ref{df:phis}.
\end{corollary}

The fact that in \eqref{eq:eta} we can group many terms with a common factor $a$ is not casual, since actually these terms come from a well-defined $\left(0,1\right)$-form with values in $T_{C_b}$

\begin{lemma} \label{lem:def-mu}
For fixed $t=b$, the expression
\begin{equation} \label{eq:mu-local}
\left(Z\frac{\partial^2Z}{\partial\overline{x}\partial x}+\frac{\partial^2Z}{\partial\overline{x}\partial t}-\ablt{Z}{\overline{x}}\ablt{Z}{x}\right)d\overline{x}\otimes\abl{x}
\end{equation}
defines a global $\left(0,1\right)$-form with coefficients on $T_{C_b}$, $\mu_b\in\Gamma\left(C_b,\caA^{0,1}\left(T_{C_b}\right)\right)$.
\end{lemma}
\begin{proof}
Let $V,V'\subseteq\caC$ be two open subsets with coordinates $\left(x,t\right)$ and $\left(x',t'\right)$, where $t=t'$ is the coordinate giving the fibration $f\colon\caC\ra B$. Note that although $t=t'$, the vector fields $\abl{t}$ on $V$ and $\abl{t'}$ on $V'$ are different on $V\cap V'$, since the are defined depending on the remaining coordinates $x$ resp. $x'$. Actually, on $V\cap V'$ there are the following relations:
\begin{equation} \label{eq:change-coord-ablt}
\abl{x}=g\abl{x'},\quad \abl{t}=h\abl{x'}+\abl{t'},\quad\abl{x'}=\frac{1}{g}\abl{x}\quad\text{and}\quad\abl{t'}=\frac{-h}{g}\abl{x}+\abl{t},
\end{equation}
where $g=\ablt{x'}{x}$ and $h=\ablt{x'}{t}$. Analogously it holds $dx'=\ablt{x'}{x}dx+\ablt{x'}{t}dt$, and thus $dx'_{\mid C_b}=\ablt{x'}{x}dx_{\mid C_b}$ for every $b\in B$. Note that $g\neq 0$ at every point (otherwise $\left(x,t\right)\mapsto\left(x',t'\right)$ would not be a change of coordinates) and also
\begin{equation} \label{eq:compatibility-g-h}
\ablt{g}{t}=\frac{\partial^2x'}{\partial x\partial t}=\ablt{h}{x}.
\end{equation}
Let also $\chi=Z\abl{x}+\abl{t}=Z'\abl{x'}+\abl{t'}$ be the local expressions of $\chi$ in $V\cap V'$ with respect to both systems of coordinates. From \eqref{eq:change-coord-ablt} it follows that
\begin{equation} \label{eq:change-coord-Z}
Z'=Zg+h.
\end{equation}
From \eqref{eq:change-coord-Z}, using that $g$ and $h$ are holomorphic and conjugating \eqref{eq:change-coord-ablt}, we obtain
$$\ablt{Z'}{\overline{x'}}=g\ablt{Z}{\overline{x'}}=\frac{g}{\overline{g}}\ablt{Z}{\overline{x}},$$
and thus using \eqref{eq:change-coord-ablt} and \eqref{eq:change-coord-Z} again
\begin{align}
Z'\frac{\partial^2Z'}{\partial x'\partial\overline{x'}}&=\frac{Zg+h}{g}\abl{x'}\left(\frac{g}{\overline{g}}\ablt{Z}{\overline{x}}\right)=\frac{Zg+h}{g\overline{g}}\left(g\frac{\partial^2Z}{\partial x\partial\overline{x}}+\ablt{g}{x}\ablt{Z}{\overline{x}}\right),\\
\frac{\partial^2Z'}{\partial t'\partial\overline{x'}}&=\frac{-h}{g\overline{g}}\abl{x}\left(g\ablt{Z}{\overline{x}}\right)+\frac{1}{\overline{g}}\abl{t}\left(g\ablt{Z}{\overline{x}}\right),\\
&=\frac{-h}{g\overline{g}}\left(\ablt{g}{x}\ablt{Z}{\overline{x}}+g\frac{\partial^2Z}{\partial x\partial\overline{x}}\right)+\frac{1}{\overline{g}}\left(\ablt{g}{t}\ablt{Z}{\overline{x}}+g\frac{\partial^2Z}{\partial t\partial\overline{x}}\right),\\
\ablt{Z'}{x'}\ablt{Z'}{\overline{x'}}&=\frac{1}{g}\ablt{\left(Zg+h\right)}{x}\frac{g}{\overline{g}}\ablt{Z}{\overline{x}}=\frac{1}{\overline{g}}\ablt{Z}{\overline{x}}\left(\ablt{Z}{x}g+Z\ablt{g}{x}+\ablt{h}{x}\right).
\end{align}
Summing up and cancelling repeated terms using \eqref{eq:compatibility-g-h}, we obtain
$$Z'\frac{\partial^2Z'}{\partial x'\partial\overline{x'}}+\frac{\partial^2Z'}{\partial t'\partial\overline{x'}}-\ablt{Z'}{x'}\ablt{Z'}{\overline{x'}}=\frac{g}{\overline{g}}\left(Z\frac{\partial^2Z}{\partial x\partial\overline{x}}+\frac{\partial^2Z}{\partial t\partial\overline{x}}-\ablt{Z}{x}\ablt{Z}{\overline{x}}\right).$$
Combining this with the fact that $d\overline{x'}\otimes\abl{x'}_{\mid C_b}=\overline{g}d\overline{x}\otimes\frac{1}{g}\abl{x}_{\mid C_b}$, it turns out that 
$$\left(Z'\frac{\partial^2Z'}{\partial x'\partial\overline{x'}}+\frac{\partial^2Z'}{\partial t'\partial\overline{x'}}-\ablt{Z'}{x'}\ablt{Z'}{\overline{x'}}\right)d\overline{x'}\otimes\abl{x'}_{\mid C_b}=\left(Z\frac{\partial^2Z}{\partial x\partial\overline{x}}+\frac{\partial^2Z}{\partial t\partial\overline{x}}-\ablt{Z}{x}\ablt{Z}{\overline{x}}\right)d\overline{x}\otimes\abl{x}_{\mid C_b}.$$
Thus the expression defining $\mu$ does not depend on the choice of coordinates $\left(x,t\right)$, as wanted.
\end{proof}

\begin{definition} \label{df:mu}
We call the class $\left[\mu_b\right]\in H^1_{\overline{\partial}}\left(T_{C_b}\right)$ \emph{second-order Kodaira-Spencer class} of $C_b\subseteq\caC$.
\end{definition}

\begin{remark} \label{remark:mu}
Here some comments on the classes $\mu_b$.
\begin{enumerate}
\item There is the following very good motivation to the name: as already mentioned in Remark \ref{rmk:KS-again}, the first-order Kodaira-Spencer class of $C_b\subseteq\caC$ can be computed as $\left[\overline{\partial}\chi_{\mid C_b}\right]\in H^1_{\overline{\partial}}\left(T_{C_b}\right)$, which depends only on the first-order infinitesimal neighbourhood of $C_b$ in $\caC$. Thus $\mu$ depends only on the second-order infinitesimal neighbourhood of $C_b$ in $\caC$, because of the term $\frac{\partial^2Z}{\partial\overline{x}\partial t}$.
\item The classes $\mu_b$ have been defined pointwise. Although their definition suggest that they depend smoothly on $b\in B$ (or even possibly holomorphically, as the classes $\xi_b$ do), it is not yet clear to us how to prove this. Indeed the first-order Kodaira-Spencer classes $\xi_b$ can be gathered into a holomorphic section of $R^1f_*T_{\caC/B}$ given by the connecting homomorphism of pushing-forward the short exact sequence \eqref{eq:relative-tangent-family}. It would be interesting to obtain such a description of the classes $\mu_b$.
\item Recall from the end of Section 2 that in the case the induced period map $\gamma\colon B\ra\caM_g\ra\caA_g$ (or one lifting to the Siegel upper-half space $\mH_g$) is an embedding, we can think of $\eta^{\left(2\right)}$ as the covariant derivative of the tangent vector field to $B$ along itself. We can then decompose $\eta^{\left(2\right)}=\eta_T^{\left(2\right)}+\eta_N^{\left(2\right)}$, where $\eta_T^{\left(2\right)}$ is tangent to $\caM_g$, and $\eta_N^{\left(2\right)}$ is normal. Since $T_{\left[C_b\right]}\caM_g=H^1\left(T_{C_b}\right)$, the $\mu_b$ also form a (possibly not even continuous) vector field along $B$ tangent to $\caM_g$, and it would be very interesting to compare both vector fields $\mu$ and $\eta_T^{\left(2\right)}$.
\end{enumerate}
\end{remark}

We will now focus our attention on the class $\eta^{\left(2\right)}\left(\alpha\right)_b-\mu_b\cdot\alpha_b\in H^1\left(\caO_{C_b}\right)\cong H^0\left(\omega_{C_b}\right)^{\vee}$, the remaining term of \eqref{eq:eta}, locally represented by $-2\left(\ablt{\varphi}{x}\ablt{Z}{\overline{x}}d\overline{x}_{\mid C_b}\right)$. More precisely, we consider its action on any form $\alpha'_b\in\caK_b$.

\begin{lemma} \label{lem:remaining-term-on-Kb}
For any $\alpha'_b\in\ker\left(\xi_b\cdot\colon H^0\left(\omega_{C_b}\right)\ra H^1\left(\caO_{C_b}\right)\right)$ it holds
$$\left(\eta^{\left(2\right)}\left(\alpha\right)_b-\mu_b\cdot\alpha_b\right)\cdot\alpha'_b=0\in H^1\left(\omega_{C_b}\right)\cong\mC.$$
\end{lemma}
\begin{proof}
Let $\alpha'_b$ be locally represented by $a\left(x\right)dx$, and recall that $\xi_b$ is represented locally by $\left(\ablt{Z}{\overline{x}}\abl{x}\otimes d\overline{x}\right)_{\mid C_b}$. Considering Dolbeaut cohomology, it holds $\xi_b\cdot\alpha'_b=0$ if and only if there is a global smooth function $\gamma\colon C_b\ra\mC$ such that $\overline{\partial}\gamma=\left(\ablt{Z}{\overline{z}}ad\overline{x}\right)_{\mid C_b}$. In this case we can write 
$$\left(\eta^{\left(2\right)}\left(\alpha\right)_b-\mu_b\cdot\alpha_b\right)\cdot\alpha'_b=-2\left(\ablt{\varphi}{x}\ablt{Z}{\overline{z}}ad\overline{x}\wedge dx\right)=2\left[\partial\varphi\wedge\overline{\partial}\gamma\right]\in H^1\left(\omega_{C_b}\right)\cong H^2_{dR}\left(C_b,\mC\right).$$
Now it is clear that $2\partial\varphi\wedge\overline{\partial}\gamma=d\left(\varphi\overline{\partial}\gamma-\gamma\partial\varphi\right)$ is $d$-exact, and the claim follows.
\end{proof}

With Lemma \ref{lem:remaining-term-on-Kb} at hand, we can now state the main result of this section regarding additional conditions defining the fibres of $\caU$.

\begin{theorem} \label{thm:mu-hat}
For any $b\in B$ let $\caK_b\subseteq E_b=H^0\left(\omega_{C_b}\right)$ be the fibre of $\caK$ on $B$, and $\mu_b\in H^1\left(T_{C_b}\right)$ the second-order Kodaira-Spencer class of $C_b\subseteq\caC$. Let
$$\widehat{\mu_b}\colon H^0\left(\omega_{C_b}\right)\stackrel{\mu_b\cdot}{\longra}H^1\left(\caO_{C_b}\right)=E_b^{\vee}\twoheadrightarrow\caK_b^{\vee}.$$
Then $\caU_b\subseteq\caK_b\cap\ker\widehat{\mu_b}$.
\end{theorem}
\begin{proof}
Let $\xi_b\in H^1\left(T_{C_b}\right)$ be the first-order Kodaira-Spencer class at $b$. In general we have $\caK_b\subseteq\ker\left(\xi_b\cdot\right)$. Hence we can apply Lemma \ref{lem:remaining-term-on-Kb} to show that $\hat{\mu_b}$ coincides with the composition
$$H^0\left(\omega_{C_b}\right)\stackrel{\eta^{\left(1\right)}_b}{\longra}H^1\left(\caO_{C_b}\right)=E_b^{\vee}\twoheadrightarrow\caK_b^{\vee}.$$
Finally from Theorem \ref{thm:U-with-etas} we have
$$\caU_b\subseteq\caK_b\cap\ker\eta^{\left(1\right)}_b\subseteq\caK_b\cap\ker\hat{\mu_b}.$$
\end{proof}

\section{Families of plane curves}

In this section we consider the special case of a smooth family of plane curves over a disk $0\in B\subset\mC$. We find a more explicit formula for $\mu$ in terms of the defining polynomials and in particular we see that at a general point $\mu_b$ and $\xi_b$ can be different, hence $\mu$ induces a non-trivial additional condition.

We start fixing some notations and recalling general facts about the tangent and normal bundles of smooth plane curves. Recall that the tangent bundle of $\mP^2$ fits into the Euler exact sequence
$$0 \longra \caO_{\mP^2} \longra \bigoplus_{i=0}^2\caO_{\mP^2}\left(1\right)\abl{Y_i} \longra T_{\mP^2} \longra 0,$$
where the first map is given by the Euler relation $1\mapsto Y_0\abl{Y_0}+Y_1\abl{Y_1}+Y_2\abl{Y_2}$. In particular $T_{\mP^2}$ is globally generated by the sections $Y_j\abl{Y_i}$ with $i,j=0,1,2$. Moreover on the standard open set $U_j=\left\{\left[a_0:a_1:a_2\right]\in\mP^2\mid a_j\neq 0\right\}\cong\mC^2$ with affine coordinates $y_{i/j}:=\frac{Y_i}{Y_j}$ ($i\neq j$) we have the identification $\abl{y_{i/j}}=Y_j\abl{Y_i}$.

Let $F\left(Y_0,Y_1,Y_2\right)\in\mC\left[Y_0,Y_1,Y_2\right]$ be a homogeneous polynomial of degree $d$ defining a smooth projective curve
$$C=V\left(F\right)=\left\{\left[a_0:a_1:a_2\right]\mid F\left(a_0,a_1,a_2\right)=0\right\}\subseteq\mP^2,$$
i.e. such that 
$$V\left(F_0,F_1,F_2\right)=\left\{\left[a_0:a_1,a_2\right]\in\mP^2\mid F_0\left(a_0,a_1,a_2\right)=F_1\left(a_0,a_1,a_2\right)=F_2\left(a_0,a_1,a_2\right)=0\right\}=\emptyset,$$
where for $i=0,1,2$ we denote by $F_i=\ablt{F}{Y_i}$.

The normal bundle satisfies $N_{C/\mP^2}\cong\caO_C\left(d\right)$, and fits in the exact sequence
\begin{equation} \label{eq:normal-bundle-plane}
0\longra T_C\longra T_{\mP^2\mid C}\stackrel{\nu}{\longra} \caO_C\left(d\right)\longra 0,
\end{equation}
where $\nu$ is given by
$$\abl{y_{i/j}}_{\mid C}=Y_j\abl{Y_i}_{\mid C}\mapsto Y_j\ablt{F}{Y_i}_{\mid C}={Y_jF_i}_{\mid C}\in H^0\left(C,\caO_C\left(d\right)\right).$$

We now consider the cover of $C$ by the open subsets
$$V_{i/j}:=\left\{\left[a_0:a_1:a_2\right]\in C\mid a_j\neq 0, F_k\left(a_0,a_1,a_2\right)\neq 0\right\}\subseteq C$$
where $k$ denotes the remaining index, i.e. $\left\{i,j,k\right\}=\left\{0,1,2\right\}$. The implicit function theorem applied to $V_{i/j}\subseteq U_j\cong \mC^2$ implies:
\begin{enumerate}
\item $x_{i/j}:=\left(y_{i/j}\right)_{\mid V_{i/j}}$ is a local coordinate for $C$ around every point $P\in V_{i/j}$ 
\item $T_{C\mid V_{i/j}}$ is generated by
\begin{equation} \label{eq:generator-T-C}
\abl{x_{i/j}}:=\left(\abl{y_{i/j}}-\frac{F_i}{F_k}\abl{y_{k/j}}\right)_{\mid C}=Y_j\left(\abl{Y_i}-\frac{F_i}{F_k}\abl{Y_k}\right)_{\mid C},
\end{equation}
\item and $N_{C/\mP^2\mid V_{i/j}}$ is generated by $\abl{y_{k/j}}_{\mid C}$.
\end{enumerate}

Note that, although $x_{i/j}$ is in general not a global coordinate on $V_{i/j}$, the associated derivation at every $P\in V_{i/j}$ defines the vector field \eqref{eq:generator-T-C} defined on the whole $V_{i/j}$.

The first connecting homomorphism of the long exact sequence of cohomology of \eqref{eq:normal-bundle-plane} gives a linear map
\begin{equation} \label{eq:delta}
\delta\colon H^0\left(N_{C/\mP^2}\right)\cong\mC\left[Y_0,Y_1,Y_2\right]_d/\left(F\right)\twoheadrightarrow H^1\left(T_C\right),
\end{equation}
whose kernel is the degree $d$ part of the jacobian ideal $J\left(F\right):=\left(F_0,F_1,F_2\right)\subseteq\mC\left[Y_0,Y_1,Y_2\right]$.

If $G\in\mC\left[Y_0,Y_1,Y_2\right]_d$ represents a global section $G_{\mid C}$ of $H^0\left(N_{C/\mP^2}\right)$, we want to give a more explicit description of $\delta G_{\mid C}\in H^1\left(T_C\right)$ in terms of Dolbeaut cohomology. To this aim note first that on $V_{i/j}$ $G_{\mid C}$ can be lifted to $\frac{G}{Y_jF_k}\abl{y_{k/j}}\in H^0\left(V_{i/j},T_{\mP^2\mid C}\right)$. Thus if $\left\{\rho_{i/j}\right\}_{i\neq j}$ is a $\caC^{\infty}$ partition of unity such that $\supp\rho_{i/j}\subseteq V_{i/j}$ for every $i\neq j$, we can lift $G_{\mid C}$ to the $\caC^{\infty}$ section 
$$\sigma=\sum_{i,j}\frac{\rho_{i/j}G}{Y_jF_k}\abl{y_{k/j}}\in\caA\left(T_{\mP^2\mid C}\right),$$
and thus $\delta G_{\mid C}$ is represented by 
\begin{equation} \label{eq:delta-G}
\overline{\partial}\sigma=\sum_{i,j}\overline{\partial}\rho_{i/j}\frac{G}{Y_jF_k}\abl{y_{k/j}}=\sum_{i,j}\overline{\partial}\rho_{i/j}\frac{G}{F_k}\abl{Y_k}_{\mid C}\in\caA^{0,1}\left(T_{\mP^2\mid C}\right).
\end{equation}
It can be quickly shown that $\overline{\partial}\sigma\in\caA^{0,1}\left(T_C\right)$ because 
$$\nu\left(\overline{\partial}\sigma\right)=\sum_{i,j}\overline{\partial}\rho_{i/j}\frac{G}{F_k}\ablt{F}{Y_k}_{\mid C}=G\overline{\partial}\left(\sum_{i,j}\rho_{i/j}\right)=0.$$
However we would like to have a more explicit description of $\overline{\partial}\sigma$ in terms of the generator $\abl{x_{i/j}}$ of $T_C$ on $V_{i/j}$. For this we set
\begin{equation} \label{eq:rho}
\rho_{ij}=\rho_{i/j}+\rho_{j/i}.
\end{equation}

\begin{lemma} \label{lem:delta-G}
The class $\delta G_{\mid C}\in H^1\left(T_C\right)$ is represented by the form given in the open subset $V_{i/j}$ as
$$\overline{\partial}\sigma=\frac{G}{Y_j}\left(\frac{1}{F_i}\overline{\partial}\rho_{k,j}-\frac{Y_i}{Y_jF_j}\overline{\partial}\rho_{i,k}\right)\abl{x_{i/j}}.$$
\end{lemma}
\begin{proof}
For the sake of simplicity we consider the case $i=0,j=2$. It is enough to show that
$$\sum_{i\neq j}\frac{\overline{\partial}\rho_{i/j}}{F_k}\abl{Y_k}_{\mid C}=\frac{1}{Y_2}\left(\frac{1}{F_0}\overline{\partial}\rho_{1,2}-\frac{Y_0}{Y_2F_2}\overline{\partial}\rho_{0,1}\right)\abl{x_{0/2}},$$
or grouping the terms with same $k$, that
$$\sum_{i<j}\frac{\overline{\partial}\rho_{i,j}}{F_k}\abl{Y_k}_{\mid C}=\frac{1}{Y_2}\left(\frac{1}{F_0}\overline{\partial}\rho_{1,2}-\frac{Y_0}{Y_2F_2}\overline{\partial}\rho_{0,1}\right)\abl{x_{0/2}}.$$
The Euler relation gives, $\abl{Y_2}=\frac{-1}{Y_2}\left(Y_0\abl{Y_0}+Y_1\abl{Y_1}\right)$, so that
\begin{align*}
\sum_{i<j}\frac{\overline{\partial}\rho_{i,j}}{F_k}\abl{Y_k}_{\mid C}&=\frac{\overline{\partial}\rho_{0,1}}{F_2}\abl{Y_2}_{\mid C}+\frac{\overline{\partial}\rho_{0,2}}{F_1}\abl{Y_1}_{\mid C}+\frac{\overline{\partial}\rho_{1,2}}{F_0}\abl{Y_0}_{\mid C}\\
&=\left(\frac{\overline{\partial}\rho_{1,2}}{F_0}-\frac{Y_0}{Y_2}\frac{\overline{\partial}\rho_{0,1}}{F_2}\right)\abl{Y_0}_{\mid C}+\left(\frac{\overline{\partial}\rho_{0,2}}{F_1}-\frac{Y_1}{Y_2}\frac{\overline{\partial}\rho_{0,1}}{F_2}\right)\abl{Y_1}_{\mid C}.
\end{align*}
Using now that $\abl{x_{0/2}}=Y_2\left(\abl{Y_0}-\frac{F_0}{F_1}\abl{Y_1}\right)_{\mid C}$ and $\rho_{0,2}+\rho_{1,2}=1-\rho_{0,1}$, we can write
\begin{align*}
\sum_{i<j}\frac{\overline{\partial}\rho_{i,j}}{F_k}\abl{Y_k}_{\mid C}&=\left(\frac{\overline{\partial}\rho_{1,2}}{F_0}-\frac{Y_0}{Y_2}\frac{\overline{\partial}\rho_{0,1}}{F_2}\right)\frac{1}{Y_2}\abl{x_{0/2}}-\overline{\partial}\rho_{0,1}\left(\frac{1}{F_1}+\frac{Y_1}{Y_2F_2}+\frac{Y_0F_0}{Y_2F_1F_2}\right)\abl{Y_1}_{\mid C}\\
&=\left(\frac{\overline{\partial}\rho_{1,2}}{F_0}-\frac{Y_0}{Y_2}\frac{\overline{\partial}\rho_{0,1}}{F_2}\right)\frac{1}{Y_2}\abl{x_{0/2}}-\overline{\partial}\rho_{0,1}\frac{Y_2F_2+Y_1F_1+Y_0F_0}{Y_2F_1F_2}\abl{Y_1}_{\mid C}\\
&=\left(\frac{\overline{\partial}\rho_{1,2}}{F_0}-\frac{Y_0}{Y_2}\frac{\overline{\partial}\rho_{0,1}}{F_2}\right)\frac{1}{Y_2}\abl{x_{0/2}},
\end{align*}
where in the last equality the Euler identity $Y_2F_2+Y_1F_1+Y_0F_0=dF$ has been used to prove the vanishing of the last summand.
\end{proof}

%\ToDo{Check wether the rho's are constant where the denominators different from $Y_j$ and $F_k$ vanish, hence the expression makes sense.}

We come now to families. Let $B\subseteq\mC$ be the unit disk and $R=\caO_{\mC}\left(B\right)$ the ring of holomorphic functions on $B$. A family of plane curves of degree $d$ over $B$ is defined by $F\left(Y_0,Y_1,Y_2,T\right)\in R\left[Y_0,Y_1,Y_2\right]_d$, a homogeneous polynomial of degree $d$ on $Y_0,Y_1,Y_2$, and where $T$ denotes the coordinate of $B$. The family is then defined by
$$S=\left\{\left(\left[a_0:a_1:a_2\right],b\right)\in\mP^2\times B\mid F\left(a_0,a_1,a_2,b\right)=0\right\}\subseteq\mP\times B,$$
with $\pi\colon S\ra B$ given by the projection onto the second factor.

We assume moreover that for every $b\in B$ the polynomial $F\left(Y_0,Y_1,Y_2,b\right)\in\mC\left[Y_0,Y_1,Y_2\right]$ defines a smooth curve $C_b$ in $\mP^2\times\left\{b\right\}$, i.e
$$\left\{\left(\left[a_0:a_1:a_2\right],b\right)\in\mP^2\times B\mid F_0\left(a_0,a_1,a_2,b\right)=F_1\left(a_0,a_1,a_2,b\right)=F_2\left(a_0,a_1,a_2,b\right)=0\right\}=\emptyset,$$
In analogy with the previous notation, we set $F_T=\ablt{F}{T}$, $F_{ij}=\ablt{F_i}{Y_j}=\frac{\partial^2F}{\partial Y_i\partial Y_j}$, $F_{Ti}=\ablt{F_T}{Y_i}=\frac{\partial^2F}{\partial T\partial Y_i}$ and $F_{TT}=\ablt{F_T}{T}=\frac{\partial^2F}{\partial T^2}$.

We extend the previous notation for the open subsets $U_i=\left\{\left(\left[a_0:a_1:a_2\right],b\right)\in\mP^2\times B\mid a_i\neq 0\right\}\subseteq\mP^2\times B$ and
$$V_{i/j}=\left\{\left(\left[a_0:a_1:a_2\right],b\right)\in S\mid a_j\neq 0, F_k\left(a_0,a_1,a_2,b\right)\neq 0\right\}.$$
The implicit function theorem shows that
\begin{enumerate}
\item $x_{i/j}:=\frac{Y_i}{Y_j}_{\mid V_{i/j}}$ and $t_{i/j}:=T_{\mid V_{i/j}}$ are local coordinates around every $P\in V_{i/j}$, and the projection $\pi\colon S\ra B$ is given by $t_{i/j}$.
\item $T_{S\mid V_{i/j}}$ is generated by
\begin{equation} \label{eq:generator-T-S-x}
\abl{x_{i/j}}:=\abl{y_{i/j}}-\frac{F_i}{F_k}\abl{y_{k/j}}_{\mid S}=Y_j\left(\abl{Y_i}-\frac{F_i}{F_k}\abl{Y_k}\right)_{\mid S},
\end{equation}
and
\begin{equation} \label{eq:generator-T-S-t}
\abl{t_{i/j}}:=\abl{T}-\frac{F_T}{Y_jF_k}\abl{y_{k/j}}_{\mid S}=\left(\abl{T}-\frac{F_T}{F_k}\abl{Y_k}\right)_{\mid S}.
\end{equation}
\item For each $b\in B$, $T_{C_b\mid V_{i/j}}$ and $N_{C_b/S\mid V_{i/j}}$ are generated by $\abl{x_{i/j}}_{\mid C_b}$ and $\abl{t_{i/j}}_{\mid C_b}$ respectively.
\end{enumerate}

Considering again a partition of unity $\left\{\rho_{i/j}\right\}_{i\neq j}$ with $\supp\rho_{i/j}\subseteq V_{i/j}$, we can construct a $\caC^{\infty}$ vector field $\chi$ on $S$ such that $d\pi\left(\chi\right)=\abl{T}$ as
\begin{equation} \label{eq:chi-plane}
\chi=\sum_{i\neq j}\rho_{i/j}\abl{t_{i/j}}=\sum_{i\neq j}\rho_{i/j}\left(\abl{T}-\frac{F_T}{Y_jF_k}\abl{Y_k}\right)_{\mid S}.
\end{equation}
In order to compute the local expressions of $\mu$ in the $V_{i/j}$ with this choice of $\chi$, we first need to compute the local expressions of $\chi$, i.e. the functions $Z_{i/j}$ such that $\chi_{\mid V_{i/j}}=Z_{i/j}\abl{x_{i/j}}+\abl{t_{i/j}}$. As in the case of one curve, we set $\rho_{i,j}=\rho_{i/j}+\rho_{j/i}$ for $i\neq j$.

\begin{lemma} \label{lem:Z-i-j}
For any $i\neq j$ it holds $\chi_{\mid V_{i/j}}=Z_{i/j}\abl{x_{i/j}}+\abl{t_{i/j}}$ with 
\begin{equation} \label{eq:Z-i-j}
Z_{i/j}=\frac{F_T}{Y_j}\left(\frac{\rho_{k,j}}{F_i}-\frac{Y_i}{Y_j}\frac{\rho_{i,k}}{F_j}\right).
\end{equation}
\end{lemma}
\begin{proof}
As in the proof of Lemma \ref{lem:delta-G}, for the sake of simplicity we show the case $i=0,j=2$.

From \eqref{eq:chi-plane} we have
\begin{equation*}
\chi=\sum_{i\neq j}\rho_{i/j}\abl{t_{i/j}}=\sum_{i\neq j}\rho_{i/j}\left(\abl{T}-\frac{F_T}{Y_jF_k}\abl{Y_k}\right)_{\mid S}=\abl{T}_{\mid S}-F_T\left(\frac{\rho_{1,2}}{F_0}\abl{Y_0}+\frac{\rho_{0,2}}{F_1}\abl{Y_1}+\frac{\rho_{0,1}}{F_2}\abl{Y_2}\right)_{\mid S}.
\end{equation*}
Using the Euler relation $\abl{Y_2}=\frac{-1}{Y_2}\left(Y_0\abl{Y_0}+Y_1\abl{Y_1}\right)$ and grouping conveniently, we obtain
\begin{align*}
\chi&=\abl{T}_{\mid S}-F_T\left[\left(\frac{\rho_{1,2}}{F_0}-\frac{Y_0}{Y_2}\frac{\rho_{0,1}}{F_2}\abl{Y_2}\right)\abl{Y_0}+\left(\frac{\rho_{0,2}}{F_1}-\frac{Y_1}{Y_2}\frac{\rho_{0,1}}{F_2}\abl{Y_2}\right)\abl{Y_1}\right]_{\mid S}\\
&=\left(\abl{T}-\frac{F_T}{F_1}\abl{Y_1}\right)_{\mid S}-F_T\left(\frac{\rho_{1,2}}{F_0}-\frac{Y_0}{Y_2}\frac{\rho_{0,1}}{F_2}\abl{Y_2}\right)\left(\abl{Y_0}-\frac{F_0}{F_1}\abl{Y_1}\right)_{\mid S}\\
&\quad+F_T\left[\frac{1}{F_1}-\frac{F_0}{F_1}\left(\frac{\rho_{1,2}}{F_0}-\frac{Y_0}{Y_2}\frac{\rho_{0,1}}{F_2}\right)-\left(\frac{\rho_{0,2}}{F_1}-\frac{Y_1}{Y_2}\frac{\rho_{0,1}}{F_2}\right)\right]\abl{Y_1}_{\mid S}\\
&=\abl{t_{0/2}}-\frac{F_T}{Y_2}\left(\frac{\rho_{1,2}}{F_0}-\frac{Y_0}{Y_2}\frac{\rho_{0,1}}{F_2}\abl{Y_2}\right)\abl{x_{0/2}}\\
&\quad+F_T\left[\frac{1}{F_1}-\frac{F_0}{F_1}\left(\frac{\rho_{1,2}}{F_0}-\frac{Y_0}{Y_2}\frac{\rho_{0,1}}{F_2}\right)-\left(\frac{\rho_{0,2}}{F_1}-\frac{Y_1}{Y_2}\frac{\rho_{0,1}}{F_2}\right)\right]\abl{Y_1}_{\mid S}.
\end{align*}
Thus it just remains to show that the last summand vanishes, which follows again using the identities $1-\rho_{1,2}-\rho_{0,2}=\rho_{0,1}$ and $Y_0F_0+Y_1F_1+Y_2F_2=dF$, since
\begin{align*}
\left[\frac{1}{F_1}-\frac{F_0}{F_1}\left(\frac{\rho_{1,2}}{F_0}-\frac{Y_0}{Y_2}\frac{\rho_{0,1}}{F_2}\right)-\left(\frac{\rho_{0,2}}{F_1}-\frac{Y_1}{Y_2}\frac{\rho_{0,1}}{F_2}\right)\right]_{\mid S}&=\left[\frac{1-\rho_{1,2}-\rho_{0,2}}{F_1}+\frac{F_0}{F_1}\frac{Y_0}{Y_2}\frac{\rho_{0,1}}{F_2}+\frac{Y_1}{Y_2}\frac{\rho_{0,1}}{F_2}\right]_{\mid S} \\
&=\rho_{0,1}\frac{Y_2F_2+Y_0F_0+Y_1F_1}{Y_2F_1F_2}_{\mid S}=0
\end{align*}.
\end{proof}

\begin{remark} \label{rmk:KS-plane}
Note that at any $b\in B$ the Kodaira-Spencer class $KS_b\left(\abl{T}_b\right)$ is represented by $\overline{\partial}\chi_{\mid C_b}$, which on $V_{i/j}$ is given by 
$$\overline{\partial}Z_{i/j}\abl{x_{i/j}}_{\mid C_b}=\frac{F_T}{Y_j}\left(\frac{\overline{\partial}\rho_{k,j}}{F_i}-\frac{Y_i}{Y_j}\frac{\overline{\partial}\rho_{i,k}}{F_j}\right)\abl{x_{i/j}}_{\mid C_b}.$$
Lemma \ref{lem:delta-G} shows then that $KS_b\left(\abl{T}_b\right)=\delta F_{T\mid C_b}$ is given by the first derivative of the defining equation on the parameter $T$, as expected.
\end{remark}

We are now ready to compute the local expression of $\mu_b$ for any $b\in B$.

\begin{theorem} \label{thm:mu-plane}
For any $b\in B$, the form $\mu_b\in\caA^{0,1}\left(T_{C_b}\right)$ is represented on $V_{0/2}$ by an expression
\begin{equation} \label{eq:mu-plane}
\left(\frac{F_{TT}}{Y_2}\left[\frac{\overline{\partial}\rho_{1,2}}{F_0}-\frac{Y_0\overline{\partial}\rho_{0,1}}{Y_2F_2}\right]+A\right)\abl{x_{0/2}},
\end{equation}
where the function $A$ depends only on first and second derivatives of $F$ different from $F_{TT}$.
\end{theorem}
\begin{proof}
Expand the expression \eqref{eq:mu-local} with $Z$ as in \eqref{eq:Z-i-j} with $i=0$ and $j=2$.
\end{proof}

Thus the second-order Kodaira-Spencer class depends explicitly on $\delta F_{TT\mid C_b}$, and other terms involving the first-order Kodaira-Spencer class and other derivatives of the defining equation. In any case, this direct dependancy on $F_{TT}$ shows that $\mu_b$ is in general independent of the first-order Kodaira-Spencer class, hence the condition obtained in Theorem \ref{thm:mu-hat} is non-trivial.

\bibliographystyle{alpha}
%\bibliography{biblio}

\end{document}